\documentclass{article}
\usepackage{amssymb,amsthm,amsmath,color}
\usepackage{a4}

\def\NC{{\cal N}}

\def\ZZ{{\mathbb Z}}

\def\PP{{\cal P}}
\def\RR{{\mathbb R}}
\def\rank{r}

\def\supp{\mathrm{supp}\;}
\def\Div{\mathrm{Div}}
\newcommand{\ord}[2]{\mathrm{ord}_{#1}(#2)} 
\newcommand{\divi}[1]{\mathrm{div}(#1)} 
\def\deg{\mathrm{deg}}
\def\dist{\mathrm{dist}}

\newtheorem{theorem}{Theorem}[section]
\newtheorem{lemma}[theorem]{Lemma}
\newtheorem{proposition}[theorem]{Proposition}
\newtheorem{corollary}[theorem]{Corollary}
\newtheorem{subtheorem}{Theorem}[theorem]
\newtheorem{claim1}[subtheorem]{Claim}
\theoremstyle{remark}
\newtheorem{remark}[theorem]{Remark}
\date{}
\begin{document}
\title{Rank of divisors on tropical curves}
\author{Jan Hladk{\'y}\thanks{
       DIMAP and Mathematics Institute, University of Warwick, Coventry CV4 7AL,
United Kingdom, E-mail: {\tt
        honzahladky@gmail.com}. This research was partially
        supported by grant KONTAKT ME 885.} \and Daniel
        Kr{\'a}l'\thanks{Mathematics Institute and Department of Computer
        Science, University of Warwick, Coventry CV4 7AL, United Kingdom.
        E-mail:
        {\tt D.Kral@warwick.ac.uk}. Previous affiliation: Computer Science Institute, Faculty of Mathematics and Physics, Charles University, Prague, Czech Republic.
     This research was partially supported by grant KONTAKT ME 885.} \and
         Serguei Norine\thanks{
         Department of Mathematics \& Statistics, McGill
         University, Burnside Hall, 805~Sherbrooke West,
		 Montreal, QC, H3A~2K6, Canada.
     	 E-mail: {\tt snorin@math.mcgill.ca}. The author
     	 was supported in part by NSF under Grant No.
     	 DMS-0200595 and DMS-0701033.} }

\maketitle
\begin{abstract}
We investigate, using purely
combinatorial methods, structural and algorithmic properties of linear
equivalence classes of divisors on tropical curves. In particular, we confirm a
conjecture of Baker asserting that the rank of a divisor $D$ on a
(non-metric) graph is equal to the rank of $D$ on the corresponding
metric graph, and construct an algorithm for computing the
rank of a divisor on a tropical curve.
\end{abstract}

\section{Introduction}
Tropical geometry investigates properties of tropical varieties,
objects which are commonly considered to be combinatorial
counterparts of algebraic varieties. There are several survey
articles on this recent branch of
mathematics~\cite{bib-mikhalkin2,bib-richtergebert+,bib-speyer+}. In particular,~\cite{bib-gathmann}
concentrates on topics which are particularly close to the subject of
this paper.

Tropical varieties share many important features with their
algebro-geometric analogues, and allow for a variety of algebraic,
combinatorial and geometric techniques to be used. 
We illustrate this on several important examples related to the
current paper.
\begin{itemize}
\item In~\cite{bib-baker+}, a version
of the Riemann-Roch theorem for graphs was proved by purely
combinatorial methods. Shortly afterwards Gathmann and
Kerber~\cite{bib-gathmann+} used the result to prove Riemann-Roch
theorem for tropical curves. Their contribution was a method of
approximating a tropical curve by graphs.
\item Mikhalkin and Zharkov~\cite{bib-mikhalkin+}
 gave (among others) another proof of the Riemann-Roch
theorem for tropical curves. Their approach used a combination of
algebraic and combinatorial techniques.
\item Recently, a machinery which allows one to transfer certain results from Riemann surfaces to
tropical curves has been developed in~\cite{bib-baker}. For example, Baker~\cite{bib-baker} introduced a tropical version of Weierstrass points, and proved using this machinery that every tropical curve of genus more than one contains at least one such point, a fact well known in the context of algebraic curves. Note that the
method necessarily has some limitations. Indeed, it is known that
analogues of some theorems about Riemann surfaces do not hold in the
tropical context, Pappus' Theorem being one such example~(\cite[\S 7]{bib-richtergebert+}).
\end{itemize}

In this paper, we contribute further towards the theory by proving
new structural results on divisors on tropical curves. In particular, we
confirm a conjecture of Baker~\cite{bib-baker} relating the ranks of a divisor
on a graph and on a tropical curve (see Theorem~\ref{thm-graph}), and construct an algorithm for computing the
rank of a divisor on a tropical curve (see Theorem~\ref{thm-alg}). All the
proofs in the paper are purely combinatorial. In an expanded version of this
article~\cite{bib-hladky+full} we have employed these results to obtain an
alternative proof of the Riemann-Roch theorem for tropical curves.

\subsection{Overview and notation}\label{sec-intro}
Throughout the paper, a {\em graph} $G$ is a finite connected multigraph that can
contain loops, i.e., $G$ is a pair consisting  of a set $V(G)$ of
{\em vertices} and a multiset $E(G)$ of {\em edges}, which are
unordered pairs of not necessarily distinct vertices. The degree
$\deg_G(v)$ of a vertex $v$ is the number of edges incident with it
(counting loops twice). The $k$-th \emph{subdivision} of a graph $G$ is the graph $G^k$
obtained from $G$ by replacing each edge with a path with $k$ inner
vertices.

Graphs have been considered as analogues of
Riemann surfaces in several contexts, in particular,
in~\cite{bib-baker+, bib-baker2+} in the context of linear
equivalence of divisors. In this paper we further investigate the
properties of linear equivalence classes of divisors. We primarily
concentrate on metric graphs, but let us start the exposition by
recalling the definitions and results from~\cite{bib-baker+}
related to (non-metric) graphs.

A {\em divisor} $D$ on a graph $G$ is an element of the free abelian
group $\Div(G)$ on $V(G)$. We can write each element $D \in \Div(G)$
uniquely as $$D = \sum_{v \in V(G)} D(v) (v)$$ with $D(v) \in \ZZ$.
We say that $D$ is {\em effective}, and write $D \geq 0$, if $D(v)
\geq 0$ for all $v \in V(G)$. For $D \in \Div(G)$, we define the
{\em degree} of $D$ by the formula $$\deg(D) = \sum_{v \in V(G)}
D(v).$$ Analogously, we define $$\deg^+(D)= \sum_{v\in V(G)
}\max\{0,D(v)\}.$$ For a function $f: V(G) \rightarrow \ZZ$, the {\em divisor associated to} $f$ is given by the formula
\[
\divi{f} = \sum_{v \in V(G)} \sum_{e = vw \in E(G)} \left( f(v) - f(w)
\right)(v).
\]
Divisors associated to integer-valued functions on
$V(G)$ are called \emph{principal}. An equivalence relation
$\sim$ on $\Div(G)$, is defined as $D \sim D'$, if and only if $D - D'$
is principal. We sometimes write
$\sim_G$ instead of $\sim$ when the graph is not clearly understood
from the context. For a divisor $D$, $|D|$ denotes the set of effective
divisors equivalent to it, i.e.,
\[
|D| = \{ E \in \Div(G) \; : \; E \geq 0 \textrm{ and } E \sim D \}.
\]
We refer to $|D|$ as the {\em (complete) linear system} associated
to $D$. Sometimes, we write $|D|_G$ for $|D|$ if the underlying graph $G$ is not
clear. If $D \sim D'$, we call the divisors $D$ and $D'$ {\em equivalent} (or {\em linearly equivalent}).

The {\em rank} of a divisor $D$ on a graph $G$ is defined as
\begin{equation}\label{eq:defR}
            \rank_{G}(D)= \min_{\substack{
                      E\ge 0\\
                      |D-E|=\emptyset}
             } \deg(E)-1 \; \mbox{.}
\end{equation}             
We frequently omit the subscript $G$ in $\rank_{G}(D)$ when the graph $G$ is clear from the context.
Also note that $\rank(D)$ depends
only on the linear equivalence class of $D$. In the classical case,
$r(D)$ is usually referred to as the dimension of the linear system
$|D|$. In our setting, however, we are not aware of any
interpretation of $r(D)$ as the topological dimension of a physical
space. Thus, we refer to $r(D)$ as ``the rank" rather than ``the
dimension". See Remark~1.13 of~\cite{bib-baker+} for further
discussion about similarities and differences between our definition
of $r(D)$ and the classical definition in the Riemann surface case.

The {\em canonical divisor} on $G$ is the divisor $K_G$ defined as
\[
K_G = \sum_{v \in V(G)} (\deg(v) - 2)(v).
\]
The {\em genus} of $G$ is the number $g = |E(G)| - |V(G)| + 1$. In graph theory, $g$ is called the {\em cyclomatic number} of $G$.

The following graph-theoretical analogue of the classical
Riemann-Roch theorem is one of the main results
of~\cite{bib-baker+}.

\begin{theorem}
\label{thm-graph-rr} If $D$ is a divisor on a loopless graph $G$ of
genus $g$, then
\[
r(D) - r(K_G - D) = \deg(D) + 1 - g.
\]
\end{theorem}

Let us note that while the graph-theoretical results, such as
Theorem~\ref{thm-graph-rr}, can be viewed as simply being analogous
to classical results from algebraic geometry, there exist deep
relations between the two contexts, e.g., a connection arising from the
specialization of divisors on arithmetic surfaces is explored in
\cite{bib-baker}.

Tropical geometry provides another connection between graph theory
and the theory of algebraic curves. The analogue of an algebraic curve
in tropical geometry is an \emph{(abstract) tropical curve}, which
following Mikhalkin~\cite{bib-mikhalkin}, can be considered simply
as a metric graph. A {\em metric graph} $\Gamma$ is a graph with
each edge being assigned a positive length. Each edge of a metric
graph is associated with an interval of the length assigned to the
edge with the end points of the interval identified with the end
vertices of the edge. A special type of edges are \emph{loops}, which are
edges where the two end points coincide. The points of these intervals are
referred to as {\em points} of $\Gamma$. The internal points of the interval are referred to as {\em internal} points of the edge and they form the
{\em interior} of the edge. Subintervals of these intervals are then
referred to as {\em segments}. 

This geometric representation of $\Gamma$ equips the metric graph
with a topology, in particular, we can speak about open and closed
sets. The distance $\dist_{\Gamma}(v,w)$ between two points $v$ and
$w$ of $\Gamma$ is measured in the metric space corresponding to the
geometric representation of $\Gamma$. For an edge
$e$ of $\Gamma$ and two points $x,y\in e$ we use $\dist_e(x,y)$ to
denote the distance between $x$ and $y$ measured on the edge $e$. Note that in
general, $\dist_\Gamma(x,y)$ can be strictly smaller than
$\dist_e(x,y)$.

The vertices of $\Gamma$ are called {\em branching points} and
the set of branching vertices of $\Gamma$ is denoted by
$B(\Gamma)$. We allow branching points of degree two. As
usual, we assume that the number of branching points of
$\Gamma$ is finite. 

A {\em tropical curve} is a metric graph where edges incident
with vertices of degree one (leaves) are allowed to have infinite length. Such
edges are identified with the interval $[0,\infty]$,
such that $\infty$ is identified with the vertex of degree one, and
are called {\em infinite edges}. The points corresponding to
$\infty$ are referred to as {\em unbounded ends}. The unbounded ends
are also considered to be points of the tropical curve.

The notions of genus, divisor, degree of a divisor and canonical
divisor $K_{\Gamma}$ readily translate from graphs to metric graphs
and tropical curves (with basis of the free abelian group of
divisors $\Div(\Gamma)$ being the infinite set of all the points of
$\Gamma$). In order to define linear equivalence  on $\Div(\Gamma)$,
the notion of rational function has to be adapted.

A {\em rational function} on a tropical curve $\Gamma$ is a
continuous function $f:\Gamma\to\RR\cup\{\pm\infty\}$ which is a
piecewise linear function with integral slopes on every edge. We
require that the number of linear parts of a rational function on
every edge is finite and the only points $v$ with
$f(v)=\pm\infty$ are unbounded ends.

The {\em order} $\ord{v}{f}$ of a point $v$ of $\Gamma$ with respect
to a rational function $f$ is the sum of outgoing slopes of all the
segments of $\Gamma$ emanating from $v$. In particular, if $v$ is
not a branching point of $\Gamma$ and the function $f$ does not
change its slope at $v$, $\ord{v}{f}=0$. Hence, there are only
finitely many points $v$ with $\ord{v}{f}\not=0$. Therefore, we can
associate a divisor $\divi{f}$ to the rational function $f$ by setting
$\divi{f}(v)=\ord{v}{f}$ for every point $v$ of $\Gamma$. Observe that
$\deg(\divi{f})$ is equal to zero as each linear part of $f$ with slope
$s$ contributes towards the sum defining $\deg(\divi{f})$ by $+s$ and
$-s$ (at its two boundary points). Note that $\ord{v}{f}$ need not be
zero for unbounded ends $v$.

Rational functions on tropical curves lead to a definition of
principal divisors on tropical curves. In particular, we say that
divisors $D$ and $D'$ on $\Gamma$ are {\em equivalent} and write $D
\sim D'$ if there exists a rational function $f$ on $\Gamma$ such
that $D = D'+\divi{f}$. With this notion of equivalence the linear system
and the rank of a divisor on a tropical curve are defined in the same
manner as for finite graphs above, in particular:
\[
|D| = \{ E \in \Div(\Gamma) \; : \; E \geq 0 \textrm{ and } E \sim D
\},
\]
$$\rank_{\Gamma}(D)= \min_{\substack{
                      E\ge 0,\; E \in \Div(\Gamma) \\
                      |D-E|=\emptyset}
             } \deg(E)-1 \; \mbox{.}$$
We may occasionally use $|D|_\Gamma$ for $|D|$ if the
underlying tropical curve $\Gamma$ is not clear from the context. Gathmann and
Kerber~\cite{bib-gathmann+} and, independently, Mikhalkin and Zharkov~\cite{bib-mikhalkin+} have proved the
following version of the Riemann-Roch theorem for tropical curves.

\begin{theorem}\label{thm-trop-rr}
Let $D$ be a divisor on a tropical curve $\Gamma$ of genus $g$. Then
\[
r(D) - r(K_{\Gamma} - D) = \deg(D) + 1 - g.
\]
\end{theorem}

Theorem~\ref{thm-trop-rr} is also proven in an expanded version of this
paper~\cite{bib-hladky+full}.


We prove in Section~\ref{sec-rank} the following theorem relating
the ranks of divisors on ordinary and metric graphs. Before stating
the theorem we need to introduce a definition. We say that a metric
graph $\Gamma$ {\em corresponds} to the graph $G$ if $\Gamma$ is
obtained from $G$ by setting the length of each edge of $G$ to be
equal to one.

\begin{theorem}
\label{thm-graph} Let $D$ be a divisor on a loopless graph $G$ and let
$\Gamma$ be the metric graph corresponding to $G$. Then,
$$r_G(D)=r_{\Gamma}(D).$$
\end{theorem}

The sets of effective divisors and principal divisors on $\Gamma$
are both strictly larger than the respective sets for $G$. Hence,
Theorem~\ref{thm-graph} is not a~priori obvious.

Theorem~\ref{thm-graph} implies a conjecture of Baker~\cite{bib-baker} that the
rank of a divisor on a loopless graph $G$ is the same as its rank on the graph
$G^k$, the graph where every edge of $G$ is $k$ times subdivided
(see~Corollary~\ref{cor-graph}). Gathmann and Kerber proved
in~\cite[Proposition~2.4]{bib-gathmann+} the following statement: Given a
divisor $D$ on a loopless graph $G$ there exist infinitely many subdivisions
$G'$ of $G$ such that $r_G(D)=r_{G'}(D)$. Corollary~\ref{cor-graph} therefore strengthens
quantification of~\cite[Proposition~2.4]{bib-gathmann+} by allowing all
possible subdivisions as opposed to just an infinite family of subdivisions.

We finish the paper by considering algorithmic applications of the
results established in Sections~\ref{sec-non-spec}
 and~\ref{sec-rank}, and design an algorithm for computing the rank
of divisors on tropical curves.

\subsection{The Riemann-Roch
criterion}\label{sec-comb-rr}

In this section, we recall an abstract criterion
from~\cite{bib-baker+} giving necessary and sufficient conditions
for the Riemann-Roch formula to hold (Theorem~\ref{thm-comb-rr}). By
Theorem~\ref{thm-trop-rr} we will be able to utilize these conditions in the context of tropical curves in
Section~\ref{sec-non-spec}. In an expanded version of
this paper~\cite{bib-hladky+full}, we proceed the other way around: we show
that the abstract conditions of Theorem~\ref{thm-comb-rr} are met for divisors
on tropical curves, thereby giving another proof of Theorem~\ref{thm-trop-rr}.

The setting for the results of this section is as follows. Let $X$
be a non-empty set, and let $\Div(X)$ be the free abelian group on
$X$. Elements of $\Div(X)$ are called {\em divisors} on $X$,
divisors $E$ with $E \geq 0$ are called {\em effective}. Let $\sim$
be an equivalence relation on $\Div(X)$ satisfying the following two
properties:
\begin{itemize}
\item[(E1)] If $D \sim D'$, then $\deg(D) = \deg(D')$.
\item[(E2)] If $D_1 \sim D_1'$ and $D_2 \sim D_2'$, then $D_1 + D_1' \sim D_2 + D_2'$.
\end{itemize}

As before, given $D \in \Div(X)$, define
$$
|D| = \{ E \in \Div(X) \; : \; E \geq 0 \textrm{ and } E \sim D \}
$$
and
$$\rank(D)= \min_{\substack{
                      E\ge 0\\
                      |D-E|=\emptyset
              }} \deg(E)-1 \; \mbox{.}$$

For a nonnegative integer $g$ (which will correspond to the
\emph{abstract genus} of $X$), let us define the set of
\emph{non-special divisors}
\begin{equation}\label{eq:DefNC}
\NC = \{ D \in \Div(X) \; : \; \deg(D) = g-1 \textrm{ and } |D| =
\emptyset \} \ .
\end{equation}
Some heed is needed when comparing our notion of non-special
divisors to the classic notion from the theory of Riemann
surfaces. Indeed, suppose that $Z$ is a compact Riemann
surface $Z$ of genus $g_Z$ with its canonical divisor $K_Z$. A
divisor $D$ on $Z$ is called \emph{special} whenever its
rank satisfies $\rank_Z(K_Z-D)\ge 0$. Thus, classically,
non-special divisors do not necessarily have rank of the
genus decreased by one, a property which will be guaranteed
by our later choice of the abstract genus $g$. However, when
we additionally assume that $\deg_Z(D)=g_Z-1$ then, by the
Riemann-Roch Theorem for Riemann surfaces, our
definition~\eqref{eq:DefNC} is consistent with the notion of
non-special divisors.

Finally, let $K$ be an element of $\Div(X)$ having degree $2g-2$.
The following theorem from~\cite{bib-baker+} gives necessary and
sufficient conditions for the Riemann-Roch formula to hold for
elements of $\Div(X) / \sim$.

\begin{theorem}
\label{thm-comb-rr} Define $\epsilon : \Div(X) \to \ZZ / 2\ZZ$ by
declaring that $\epsilon(D) = 0$ if $|D| \neq \emptyset$ and
$\epsilon(D) = 1$ if $|D| = \emptyset$. Then the Riemann-Roch
formula
\begin{equation*}
r(D) - r(K - D) = \deg(D) + 1 - g
\end{equation*}
holds for all $D \in \Div(X)$ if and only if the following two
properties are satisfied:
\begin{itemize}
\item[\rm{(RR1)}] For every $D \in \Div(X)$, there exists $\nu \in \NC$ such that
$$\epsilon(D) + \epsilon(\nu - D) = 1\;.$$
\item[\rm{(RR2)}] For every $D \in \Div(X)$ with $\deg(D) = g-1$, we have
$$\epsilon(D) + \epsilon(K-D) = 0\;.$$
\end{itemize}
\end{theorem}

In addition to Theorem~\ref{thm-comb-rr}, we will later use the
following lemma from~\cite{bib-baker+} that also holds in the abstract
setting.


\begin{lemma}
\label{lem-alt-rank} If (RR1) holds, then for every $D \in \Div(X)$
we have
\begin{equation*}
r(D) = \left( \min_{\substack{D' \sim D \\
\nu \in \NC} } \deg^+(D' - \nu) \right) - 1 \ .
\end{equation*}
\end{lemma}

\subsection{Reducing tropical curves to loopless metric
graphs}\label{sec-trop-reduce}

We finish the introductory part of the paper by reducing
the study of divisors on tropical curves to the
corresponding situation on loopless metric graphs. Let $\Gamma$
be a tropical curve, and let $\Gamma'$ be the metric graph obtained
from $\Gamma$ by removing interiors of infinite edges and their
unbounded ends. There exists a natural retraction map
$\psi_{\Gamma}: \Gamma \rightarrow \Gamma'$ that maps deleted points
of infinite edges of $\Gamma$ to the ends of those edges that belong
to $\Gamma'$ and is an identity on the points of $\Gamma'$. This map
induces a map from $\Div(\Gamma)$ to $\Div(\Gamma')$, which is
denoted by $\psi_\Gamma$. The following proposition combines the results of
Lemma 3.4, Remark 3.5, Lemma 3.6 and Remark 3.7
of~\cite{bib-gathmann+}.

\begin{proposition}\label{thm-trop-reduce}
Let $\Gamma$ be a tropical curve, and let $\Gamma'$ and
$\psi_{\Gamma}$ be defined as above. Let $D \in \Div(\Gamma)$, and
set $D'=\psi_{\Gamma}(D)$. We have $D \sim_{\Gamma} D'$,
$\deg(D)=\deg(D')$, and $\rank_{\Gamma}(D)=\rank_{\Gamma'}(D')$.
In addition, it holds that
$K_{\Gamma'}=\psi_{\Gamma}(K_\Gamma)$.
\end{proposition}

It follows from Proposition~\ref{thm-trop-reduce} that
Theorem~\ref{thm-trop-rr} restricted to metric graphs implies
Theorem~\ref{thm-trop-rr} in full generality. It also follows that
given an algorithm to compute the rank of divisors on metric graphs one
can readily design an algorithm to compute rank of divisors on
tropical curves. Based on these observations we concentrate our
further investigations on metric graphs.

Further, we shall restrict ourselves to \emph{loopless} metric graphs in
auxiliary lemmas leading to our main results. The general case can be reduced to
the loopless one by introducing a branching point of degree two on each edge.
This transformation does not change the set of divisors or their properties.

\subsection{Rank-determining sets}\label{ssec:RankDetermining}
The results of this paper have been substantially extended since the first
version of this manuscript was posted on the arXiv in 2007. Most importantly,
Luo~\cite{bib-luo} introduced the notion of rank-determining set, and using this
notion he extended Theorem~\ref{thm-graph}. Indeed, our Theorem~\ref{thm-graph}
asserts that the rank of a divisor on a (loopless) metric graph with edges of
integral lengths can be determined on a much simpler (and finite) object, that is, on a
graph. Luo's main result roughly speaking says that such a finitization is
possible even for general metric graphs. To state Luo's result precisely, we need
to give the definition of rank-determing sets which in turn builds on the
notion of restricted rank. Let $A$ be a non-empty set of points of
a metric graph $\Gamma$. We then define the \emph{$A$-restricted rank} of a
divisor $D\in \Div(\Gamma)$ by 
\begin{equation}\label{eq:defRR}
             \rank_A(D)= \min_{\substack{
                      E\in \Div(A), E\ge 0\\
                      |D-E|=\emptyset
              }} \deg(E)-1 \; \mbox{.}
\end{equation}              
The set $A$ is \emph{rank-determining} if $\rank_\Gamma(D)=\rank_A(D)$ for each
$D\in \Div(\Gamma)$.

We are now ready to state our main result concerning rank-determining sets, a result originally due to Luo~\cite{bib-luo}.\footnote{The proof of Theorem~\ref{thm-graph} as given in the original version of the manuscript (version 1 on the arXiv) can actually serve as a proof of Theorem~\ref{thm:rankdeterming} due to Luo if phrased in the right terminology.}
\begin{theorem}\label{thm:rankdeterming}
The set $B(\Gamma)$ of branching points of any loopless metric graph $\Gamma$ is rank-determining.
\end{theorem}
We give a proof of Theorem~\ref{thm:rankdeterming} in Section~\ref{sec-rank}. See also~\cite{bib-amini} for an alternative proof of Theorem~\ref{thm:rankdeterming}.

As we show in Section~\ref{sec-rank}, Theorem~\ref{thm:rankdeterming} implies
Theorem~\ref{thm-graph}. While not difficult, the argument is not entirely
trivial. In particular, mind that the condition $|D-E|=\emptyset$
in~\eqref{eq:defR} and in~\eqref{eq:defRR} refer to different equivalence
relations on the sets of divisors.

\section{Non-special divisors, alignments, and rank-pairs}\label{sec-non-spec}

As in Subsection~\ref{sec-comb-rr}, we define the set of
\emph{non-special divisors} on a metric graph $\Gamma$ to be
\[
\NC = \{ D \in \Div(\Gamma) \; : \; \deg(D) = g-1 \textrm{ and }
|D| = \emptyset \}
\]
where $g$ is the genus of $\Gamma$. The main results of this section is an
alternative formula for computing the rank of a divisor on a metric graph
(Corollary~\ref{cor-rank}).

We rely on results from~\cite{bib-mikhalkin+}. An alternative, self-contained
approach which gives as a by-product another (purely combinatorial) proof of the
Riemann-Roch theorem for tropical curves can be found in an expanded version of this paper~\cite{bib-hladky+full}. 

\subsection{A formula for the rank of a divisor}
%

We present a class of non-special divisors that is of
primary interest to us in our later considerations.
Let $P$ be an ordered sequence of finitely many points
of $\Gamma$. We say that the set of points in $P$ is the \emph{support} of $P$ and denote it by $\supp P$. The sequence $P$ can also be viewed as a linear
order $<_P$ on $\supp P$. If $B(\Gamma)\subseteq \supp P$ then $P$ is an \emph{alignment} of points of $\Gamma$. The set
of all alignments of points of $\Gamma$ is denoted by $\PP(\Gamma)$.

We now define a divisor $\nu_P$ corresponding to an alignment $P$.
A segment $L$ of $\Gamma$ is a \emph{$P$-segment} if
both ends of $L$ belong to $\supp P$, and the interior of $L$ is
disjoint from $\supp P$. For $v \in \supp P$, let $S_P(v)$ denote
the set of $P$-segments of $\Gamma$ with one end at $v$ and the other end
preceding $v$ in the order determined by $P$. Finally, let
$$\nu_P =\sum_{v \in\; \supp P}(|S_p(v)|-1)(v).$$ It is easy to
verify that $\deg(\nu_P)=g-1$, where $g$ is the genus of $\Gamma$. We start our investigation of divisors
corresponding to alignments by giving two simple propositions.

\begin{proposition}
\label{prop-perm} Let $P$ be an alignment of points of a metric
graph $\Gamma$. For every point $v$ of $\Gamma$ that is not contained in
$\supp P$, there exists an alignment $P'$ such that $\supp
P'=\supp P\cup\{v\}$ and $\nu_P=\nu_{P'}$.
\end{proposition}

\begin{proof}
Such an alignment $P'$ can be obtained by inserting the point $v$
in the sequence $P$ between the (distinct) boundary points of the
(unique) segment containing $v$.
\end{proof}

\begin{proposition}[{\cite[Lemma~7.8]{bib-mikhalkin+}}] \label{prop-permnoneff} If $P$ is an
alignment of points of a metric graph $\Gamma$, then $\nu_P \in \NC$.
\end{proposition}
%

The next fact asserts that
every divisor is either equivalent to an effective divisor, or is
equivalent to a divisor dominated by $\nu_P$ for some alignment
$P$, and not both.

\begin{corollary}[{\cite[Corollary~7.9]{bib-mikhalkin+}}]
\label{cor-nu-ineq} Let $\Gamma$ be a metric graph. For every $D \in
\Div(G)$, exactly one of the following holds
\begin{enumerate}
\item[(a)] $r(D) \ge 0$; or
\item[(b)] $r(\nu_P - D) \ge 0$ for some alignment $P$.
\end{enumerate}
\end{corollary}

\begin{corollary}[{\cite[Corollary~7.10]{bib-mikhalkin+}}]
\label{cor-nu-eq} If $\nu$ is a non-special divisor on a metric
graph $\Gamma$ of genus $g$, then $\nu \sim \nu_P$ for some
alignment $P$ of a finite set of points of $\Gamma$.
\end{corollary}

Corollary~\ref{cor-nu-ineq} is a consequence of
Proposition~\ref{prop-permnoneff}. Corollary~\ref{cor-nu-eq} in turn follows from Corollary~\ref{cor-nu-ineq} applied to non-special divisors. The arguments to derive these corollaries are rather straightforward. We refer the reader to~\cite{bib-mikhalkin+,bib-hladky+full}.

We finish this section with establishing a formula for rank of divisors
on metric graphs that will be central in our later analysis of the rank.

\begin{corollary}
\label{cor-rank} If $D$ is a divisor on a metric graph $\Gamma$, then the following formula holds:
\begin{equation}\label{eq:smMin}
\rank(D)=\min_{\substack{
                      D'\sim D\\
              P\in \PP(\Gamma)
              }} \deg^+ (D'-\nu_P) -1.
\end{equation}              
\end{corollary}
\begin{proof}The Riemann-Roch formula holds for divisors on metric graphs
(Theorem~\ref{thm-trop-rr}), and hence condition~(RR1) from
Theorem~\ref{thm-comb-rr} is satisfied. Lemma~\ref{lem-alt-rank} can
be applied and we infer that
\begin{equation}\label{eq:bigMin}
r(D) = \min_{\substack{D' \sim D \\
\nu \in \NC} } \deg^+(D' - \nu)  - 1 \ .
\end{equation}

By Proposition~\ref{prop-permnoneff}, the minimum
in~\eqref{eq:smMin} is taken over smaller set of parameters than
the minimum in~\eqref{eq:bigMin}. Hence, it is enough to show that there
exist $D''\sim D$ and $P \in \PP(\Gamma)$ such that $\rank(D)=\deg^+ (D''-\nu_P)-1$. Let $D' \sim D$ and $\nu \in \NC$ be chosen so that
$\rank(D)=\deg^+ (D'-\nu)-1$.

By Corollary~\ref{cor-nu-eq}, we have $\nu \sim \nu_P$ for some
alignment $P$ of points of $\Gamma$. Setting $D''=D'+(\nu_P-\nu)$
yields
$$\rank(D)=\deg^+ (D'-\nu)-1=\deg^+(D''-\nu_P)-1,$$
as desired.
\end{proof}
Motivated by Corollary~\ref{cor-rank}, we say that the pair
$(D',P)$ is a \emph{rank-pair for} $D$ if $D'\sim D$, and
$\rank(D)=\deg^+ (D'-\nu_P) -1$.

%
%

\medskip

Note that the result analogous to Corollary~\ref{cor-rank} also
holds for non-metric graphs, as shown in~\cite{bib-baker+}. Let
$\PP(G)$ denote the set of all alignments of $V(G)$. As in the
case of metric graphs, we can define the divisor $\nu_P$
corresponding to $P \in \PP(G)$ by setting $\nu_P(v)$ to be equal to
the number of edges from $v$ to vertices in $V(G)$ preceding $v$,
decreased by one. The next formula for the rank of a divisor on a
finite graph $G$ was established by Baker and
the last author~\cite{bib-baker+}.

\begin{lemma}
\label{lm-rank} The following formula holds for the rank of every
divisor $D$ on a graph $G$:
\begin{equation}\label{eq:Vecer}
\rank(D)=\min_{\substack{
                  D'\sim D\\
                  P \in \PP(G)
                 }} \deg^+ (D'-\nu_P) -1 \;\mbox{.}
\end{equation}                 
\end{lemma}
It should be remarked that Lemma~\ref{lm-rank} follows in a
straightforward fashion from the results above; at least for loopless graphs.
Indeed, we use the Riemann-Roch Theorem~\ref{thm-graph-rr} to conclude that
property~(RR1) in Theorem~\ref{thm-comb-rr} holds. Then
Lemma~\ref{lm-rank} is just Lemma~\ref{lem-alt-rank}.
In~\cite{bib-baker+}, where these results were first obtained
the arguments proceeded the other way round: the Riemann-Roch
Theorem~\ref{thm-graph-rr} was established as a consequence of
abstract combinatorial criteria from
Theorem~\ref{thm-comb-rr}.

\section{Rank of divisors on metric graphs}\label{sec-rank}

In this section we show that the divisor and the alignment in
Corollary~\ref{cor-rank} can be assumed to have a very special
structure. We establish a series of lemmas strengthening our
assumptions on this structure. It will then follow from our results
that the rank of a divisor on a graph and on the corresponding
metric graph are the same, thereby establishing
Theorem~\ref{thm-graph}.

\begin{lemma}
\label{lm-support} Let $D$ be a divisor on a loopless metric graph
$\Gamma$. Suppose there exists an
alignment $P$ of points of $\Gamma$ such that $\rank (D)=\deg^+ (D-\nu_P)-1$.
Then there also exists an alignment $P'$ of the points of $B(\Gamma)\cup \supp D$ such that
$\rank(D)=\deg^+ (D-\nu_{P'}) -1$.
\end{lemma}

\begin{proof}
By Proposition~\ref{prop-perm}, we can assume that the support of
$P$ contains all the points of $B(\Gamma)\cup\supp D$.
Choose among all alignments $P'$ satisfying $\rank(D)=\deg^+ (D-\nu_{P'})-1$
and $B(\Gamma)\cup\supp D\subseteq \supp P'$ an alignment
such that $|\supp P'|$ is minimal.

If $\supp P'=B(\Gamma)\cup\supp D$, then the lemma holds.
Assume that there exists a point $v_0\in\supp P'\setminus
(B(\Gamma)\cup\supp D)$. Let $v_1,v_2 \in \supp P'$ be
such that the segments in $\Gamma$ with ends $v_0$ and $v_i$, for $i=1,2$, contain
no other points of $\supp P'$. 
We can assume by symmetry that
$v_1<_{P'} v_2$.

Consider now the alignment $P''$ obtained from $P'$ by removing
the point $v_0$. We shall distinguish three cases based on the
mutual order of $v_0$, $v_1$ and $v_2$ in $P'$, and conclude in each
of the cases that $\deg^+(D-\nu_{P''})\le\deg^+(D-\nu_{P'})$. This,
together with the fact that $\supp P'' \subsetneq \supp P'$ will
contradict the choice of $P'$.

If $v_0<_{P'} v_1$ and $v_0<_{P'} v_2$, then $\nu_{P'}(v_0)=-1$.
Observe that $\nu_{P''}(v_1)=\nu_{P'}(v_1)-1$, $\nu_{P''}(v_0)=0$,
and $\nu_{P''}(v)=\nu_{P'}(v)$ for $v\not=v_0,v_1$. We infer that
\begin{align*}
\deg^+&(D-\nu_{P'})-\deg^+(D-\nu_{P''}) =
\\ &= 1 + \max\{D(v_1)-\nu_{P'}(v_1),0\}-\max\{D(v_1)-\nu_{P'}(v_1)+1,0\}
\ge 0\;\mbox{.}
\end{align*}
Therefore $\deg^+(D-\nu_{P''})\le\deg^+(D-\nu_{P'})$.

If $v_1<_{P'} v_0<_{P'} v_2$, then $\nu_{P'}=\nu_{P''}$ and again
$\deg^+(D-\nu_{P''})=\deg^+(D-\nu_{P'})$.

It remains to consider the case $v_1<_{P'} v_0$ and $v_2<_{P'} v_0$.
Observe that $\nu_{P'}(v_0)=1$, $\nu_{P''}(v_0)=0$,
$\nu_{P''}(v_2)=\nu_{P'}(v_2)+1$, and
$\nu_{P''}(v)=\nu_{P'}(v)$ for $v\not=v_0,v_2$.
We conclude that
\begin{align*}
\deg^+&(D-\nu_{P'})-\deg^+(D-\nu_{P''})=\\
 &=\max\{D(v_2)-\nu_{P'}(v_2),0\}-\max\{D(v_2)-\nu_{P'}(v_2)-1,0\}
  \ge 0\;\mbox{.}
\end{align*}
Consequently, $\deg^+(D-\nu_{P''})\le\deg^+(D-\nu_{P'})$.
\end{proof}

Next, we show that the divisor $D' \sim D$ that minimizes $\min_{P
\in \PP(\Gamma)} \deg^+ (D'-\nu_P)$ can be assumed to be
non-negative everywhere except for the points of $B(\Gamma)$.

\begin{lemma}
\label{lm-positive} Let $D$ be a divisor on a loopless metric graph
$\Gamma$. 
There exists a rank-pair $(D',P)$ for $D$ such that 
$P$ is an alignment of the points of
$B(\Gamma)\cup\supp D'$ and $D'$ is non-negative on the interior of every edge of $\Gamma$.
\end{lemma}

\begin{proof}
By Corollary~\ref{cor-rank} and Lemma~\ref{lm-support}, there exist
a divisor $D_0$ equivalent to $D$ and an alignment $P_0$ of the
points of $B(\Gamma)\cup \supp D_0$ such that
$\rank(D)=\deg^+ (D_0-\nu_{P_0})-1$. Among all such divisors let us consider the
divisor $D_0$ such that the sum
$$S=\sum_{v\:\in\:\supp D_0\setminus B(\Gamma)}\min\{0,D_0(v)\}$$
is maximal. If $S=0$, then the divisor $D_0$ is non-negative on the
interior of every edge of $\Gamma$, and there is nothing to prove.
Hence, we assume $S<0$ in the rest, i.e., there exists an edge $e$
with an internal point where $D_0$ is negative.

Let $v_1,\ldots,v_k$ be the longest sequence of points of $\supp D_0$ in the
interior of $e$, such that $D_0(v_i)<0$ for $i=1,\ldots,k$ and the
points are consecutive points, i.e., there is no
point of $\supp D_0$ on the segment between $v_i$ and $v_{i+1}$,
$i=1,\ldots,k-1$. Let $w_1$ be the point of $B(\Gamma)\cup \supp D_0$ such that the segment between $v_1$ and $w_1$
contains no point of $B(\Gamma) \cup\supp D_0$ and
$w_1\not=v_2$, and let $w_2$ be the point of $B(\Gamma)\cup\supp D_0$ such that the segment between $v_k$ and $w_2$
contains no point of $B(\Gamma)\cup\supp D_0$ and
$v_{k-1}\not=w_2$.

We now modify the divisor $D_0$ and the alignment $P_0$. By
symmetry, we can assume that $\dist_e(w_1,v_1)\le\dist_e(w_2,v_k)$. Let
$d_0=\dist_e(w_1,v_1)$, let $L$ be the segment from $w_1$ to $w_2$
that contains $v_1,\ldots,v_k$ and let $v'_k$ be the point of $L$ at
distance $d_0$ from $w_2$. Consider the rational function $f$ equal
to $0$ on $L$ between $v_1$ and $v'_k$ and
$f(v)=\min\{d_0,\dist_e(v,v_1),\dist_e(v,v'_k)\}$ elsewhere. Observe
that $\ord{w_1}{f}=\ord{w_2}{f}=-1$ if $w_1\not=w_2$, $\ord{w_1}{f}=-2$ if $w_1=w_2$, $\ord{v_1}{f}=\ord{v'_k}{f}= 1$ if
$v_1\not=v'_k$, $\ord{v_1}{f}=2$ if $v_1=v'_k$, and $\ord{v}{f}=0$ if $v\not=w_1,w_2,v_1,v'_k$. In addition, observe that if
$w_1=w_2$, then $L$ is a loop in $\Gamma$, and $w_1=w_2$ is a
branching point of $\Gamma$.

Let $D'_0=D_0+\divi{f}$. We first show that the sum
$$S'=\sum_{v\: \in \:\supp D'_0\setminus B(\Gamma)}\min\{0,D'_0(v)\}$$
is strictly larger than $S$. The value of $D'_0$ is smaller than the
value of $D_0$ only at $w_1$ and $w_2$. If $w_1$ is a branching
point, then the change of the value of the divisor at $w_1$ does not
affect the sum. Otherwise, the points $w_1$ and $w_2$ are distinct
(as we have observed earlier), and $D_0(w_1)\ge 1$ by the choice of
$w_1$. Hence, $D'_0(w_1)\ge 0$ and the sum is not affected by the
corresponding summand. Analogous statements are true for the point
$w_2$. We infer from $\ord{v_1}{f}>0$ that $D'_0(v_1)>D_0(v_1)$. Since
$D_0(v_1)<0$, this change increases the sum by one. Finally, the
change at $v'_k$ either increases the sum by one (if $D_0(v'_k)<0$)
or does not affect the sum (if $D_0(v'_k)\ge 0$) at all. We conclude
that $S'\ge S+1$.

We next modify the alignment $P_0$ to $P'_0$ in such a way that $\deg^+
(D_0-\nu_{P_0})=\deg^+ (D'_0-\nu_{P'_0})$. Without loss of
generality, we assume that $v'_k\in\supp P_0$
(cf.~Proposition~\ref{prop-perm}). The alignment $P'_0$ is
obtained from $P_0$ as follows: all the points of $\supp P_0$
distinct from $v_1,\ldots,v_k$ and $v'_k$ form the initial part of
the alignment in the same order as in $P_0$, and the points
$v_1,v_2,\ldots,v_k,v'_k$ then follow (in this order).

Let $W=\{w_1,w_2,v_1,\ldots,v_k,v'_k\}$. For simplicity, let us
assume that the points $w_1$ and $w_2$ are distinct, as well as
the points $v_1$, $v_k$ and $v'_k$. It is easy to
verify that all our arguments translate to the setting when some of
these points coincide. Since $D_0(v)=D'_0(v)$ and
$\nu_{P_0}(v)=\nu_{P'_0}(v)$ for all points $v\not\in W$, the
following holds:
\begin{align*}
\deg^+ &(D_0-\nu_{P_0})-\deg^+ (D'_0-\nu_{P'_0})= \\ &=\sum_{v\in
W}\left(\max\{0,(D_0-\nu_{P_0})(v)\}-
               \max\{0,(D'_0-\nu_{P'_0})(v)\}\right)\;\mbox{.}
\end{align*}
By the choice of the points $v_1,\ldots,v_k$, we have $D_0(v_i)\le
-1$ and therefore $(D_0-\nu_{P_0})(v) \le 0$ for $v \in W \setminus
\{v'_k,w_1,w_2\}$. Note also that $\nu_{P'_0}(v_i)=0$ and
$\nu_{P'_0}(v'_k) = 1$. Finally, note that $D'_0(v_i) \le D_0(v_i) +
1 \le 0$, unless $v_i = v'_k$, and $D'_0(v'_k) \le 1$. As a result,
we have $(D'_0-\nu_{P'_0})(v)\le 0$ for $v \in W \setminus
\{w_1,w_2\}$. Consequently, we obtain the following:
\begin{align*}
\deg^+ &(D_0-\nu_{P_0})-\deg^+ (D'_0-\nu_{P'_0})=\\
 =&\max\{0,(D_0-\nu_{P_0})(w_1)\}-\max\{0,(D'_0-\nu_{P'_0})(w_1)\}\\
 &+\max\{0,(D_0-\nu_{P_0})(w_2)\}-\max\{0,(D'_0-\nu_{P'_0})(w_2)\}\;\mbox{.}
\end{align*}
Since $\ord{w_1}{f}=-1$, we have $D'_0(w_1)=D_0(w_1)-1$. On the other
hand, the value $\nu_{P'_0}(w_1)$ is either equal to
$\nu_{P_0}(w_1)$, or to $\nu_{P_0}(w_1)-1$ (the latter is the case
if $w_1>_{P_0}v_1$). We conclude that $(D'_0-\nu_{P'_0})(w_1)$ is
equal to either $(D_0-\nu_{P_0})(w_1)$ or $(D_0-\nu_{P_0})(w_1)-1$.
Hence,
$$\max\{0,(D_0-\nu_{P_0})(w_1)\}-\max\{0,(D'_0-\nu_{P'_0})(w_1)\}\ge
  0\;\mbox{.}$$
An entirely analogous argument yields that
$$\max\{0,(D_0-\nu_{P_0})(w_2)\}-\max\{0,(D'_0-\nu_{P'_0})(w_2)\}\ge
  0\;\mbox{.}$$
Consequently, we obtain that
$$\deg^+ (D_0-\nu_{P_0})-\deg^+ (D'_0-\nu_{P'_0})\ge 0\;\mbox{.}$$
Since $\rank(D)=\deg ^+ (D_0-\nu_{P_0})-1$, and $D'_0$ is equivalent
to $D$, the inequality above must be the equality, and thus
$\rank(D)=\deg^+ (D'_0-\nu_{P'_0})-1$. 
By Lemma~\ref{lm-support}, there exists an alignment $P''_0$ of points of $B(\Gamma)\cup \supp D'_0$ such that $\rank(D)=\deg^+ (D'_0-\nu_{P''_0})-1$. As $S'>S$, the existence of $D'_0$ contradicts the choice of $D_0$.
\end{proof}

Next, we show that the divisor $D'$ can be assumed to be zero
outside $B(\Gamma)$, except possibly for a single point on each
edge, where its value could be equal to one.

\begin{lemma}
\label{lm-single} Let $D$ be a divisor on a loopless metric graph
$\Gamma$.
There exists rank-pair $(D',P)$ for $D$ such that 
$P$ is an alignment of the points of
$B(\Gamma)\cup \supp D'$, and every edge $e$ of $\Gamma$ contains at most one point $v$ where $D'$ is non-zero, and if such a point $v$ exists,
then $D'(v)=1$.

Furthermore, the alignment $P$ can be
assumed to be such that all the non-branching points of $\supp P$
follow the branching points in the order determined by $P$.
\end{lemma}

\begin{proof}
By Lemma~\ref{lm-positive}, there exist a divisor $D_0$ and an alignment $P_0$
of the points of $B(\Gamma)\cup \supp D_0$ such that  $(D_0,P)$ is a rank-pair for $D$, and
$D_0$ is non-negative in the interior of every edge of $\Gamma$. Among all such divisors, consider the
divisor $D_0$ such that the sum
$$S=\sum_{v\: \in \: \supp D_0\setminus B(\Gamma)}D_0(v)$$
is minimal. If every edge $e$ contains at most one point $v$ where
$D_0$ is non-zero, and $D_0(v)=1$ at such a point $v$, then the
lemma holds. We assume that $D_0$ does not have this property for a
contradiction.

Choose an edge $e$ such that the sum of the values of $D_0$ in the
interior of $e$ is at least two. Let $w_1$ and $w_2$ be the end
points of $e$ and $v_1,\ldots,v_k$ all the points of $\supp D_0$
inside $e$ ordered from $w_1$ to $w_2$. In the rest we assume that
$v_1\not=v_k$. As in the proof of the previous lemma, our
arguments readily translate to the setting when some of these points
are the same, but this assumption helps us to avoid technical
complications during the presentation of the proof. Let us note, in
order to assist the reader with the verification of the remaining
cases, that if $v_1=v_k$, then $D_0(v_1)\ge 2$.

By symmetry, we can assume that $\dist_e(w_1,v_1)\le\dist_e(w_2,v_k)$.
Let $d_0=\dist_e(w_1,v_1)$ and let $w'_2$ be the point on the segment
between $v_k$ and $w_2$ at distance $d_0$ from $v_k$. For the sake
of simplicity, we assume that $w_2\not=w'_2$; again, our arguments
readily translate to the setting when $w_2=w'_2$. Consider the
rational function $f$ equal to $0$ on the points outside the edge
$e$ and on the segment between $w_2$ and $w'_2$ and
$f(v)=\min\{\dist_e(v,w_1),\dist_e(v,w'_2),d_0\}$ elsewhere. Observe
that $\ord{w_1}{f}=\ord{w'_2}{f}=1$, $\ord{v_1}{f}=\ord{v_k}{f}=-1$, and
$\ord{v}{f}=0$ if $v\not=w_1,w'_2,v_1,v_k$.

Let $D'_0=D_0+\divi{f}$. Since $D'_0(v_1)=D_0(v_1)-1\ge 0$,
$D'_0(v_k)=D_0(v_k)-1\ge 0$ and $D'_0(w'_2)=1$, the sum
$$S'=\sum_{v\: \in \: \supp D'_0\setminus B(\Gamma)}D'_0(v)$$
is equal to $S-1$, and $D'_0$ is non-negative in the interior of all
the edges of $\Gamma$.

Next, we construct an alignment $P'_0$ such that $\rank
(D)=\deg^+ (D'_0-\nu_{P'_0})-1$. First, insert $w'_2$ into $P_0$
between $v_k$ and $w_2$, preserving the order of $v_k$ and $w_2$
(this did not change $\nu_{P_0}$, see Proposition~\ref{prop-perm}).
The alignment $P'_0$ is obtained from $P_0$ as follows: the points
$v_1,\ldots,v_k$ form the initial part of $P'_0$ in the same order
as they appear in $P_0$, and the remaining points form the final
part of $P'_0$, again in the same order as they appear in $P_0$.

It is easy to verify that $\nu_{P_0}(v)=\nu_{P'_0}(v)$ for all
points $v\not \in \{w_1,w'_2,v_1,v_k\}$. Hence,
\begin{align*}
\deg^+ &(D_0-\nu_{P_0})-\deg^+ (D'_0-\nu_{P'_0})=\\
 =&\max\{0,(D_0-\nu_{P_0})(w_1)\}-\max\{0,(D'_0-\nu_{P'_0})(w_1)\}\\
 &+\max\{0,(D_0-\nu_{P_0})(w'_2)\}-\max\{0,(D'_0-\nu_{P'_0})(w'_2)\}\\
&+\max\{0,(D_0-\nu_{P_0})(v_1)\}-\max\{0,(D'_0-\nu_{P'_0})(v_1)\}\\
&+\max\{0,(D_0-\nu_{P_0})(v_k)\}-\max\{0,(D'_0-\nu_{P'_0})(v_k)\}\;\mbox{.}
\end{align*}
Let us first consider the points $v_1$ and $w_1$.
We distinguish two cases based on the mutual order of $v_1$ and $w_1$ in $P_0$.

The case we consider first is that $v_1<_{P_0}w_1$. We have
$\nu_{P_0}(v_1)=\nu_{P'_0}(v_1)\le 0$ and
$\nu_{P_0}(w_1)=\nu_{P'_0}(w_1)\ge 0$. As $D'_0(v_1)=D_0(v_1)-1\ge
0$, we have that
$$\max\{0,(D_0-\nu_{P_0})(v_1)\}-\max\{0,(D'_0-\nu_{P'_0})(v_1)\}=1\;\mbox{.}$$
As $D'_0(w_1)=D_0(w_1)+1$, we have that
$$\max\{0,(D_0-\nu_{P_0})(v_1)\}-\max\{0,(D'_0-\nu_{P'_0})(v_1)\}\ge-1\;\mbox{.}$$
We conclude that
\begin{align*}
\max&\{0,(D_0-\nu_{P_0})(w_1)\}-\max\{0,(D'_0-\nu_{P'_0})(w_1)\}\\
 &+\max\{0,(D_0-\nu_{P_0})(v_1)\}-\max\{0,(D'_0-\nu_{P'_0})(v_1)\}\ge
0\;\mbox{.}
\end{align*}
Let us deal with the other case when $v_1>_{P_0}w_1$. Since
$\nu_{P_0}(v_1)=\nu_{P'_0}(v_1)+1$ and $D'_0(v_1)=D_0(v_1)-1$, we
have
$$\max\{0,(D_0-\nu_{P_0})(v_1)\}=\max\{0,(D'_0-\nu_{P'_0})(v_1)\}\;\mbox{.}$$
Similarly, since $\nu_{P_0}(w_1)=\nu_{P'_0}(w_1)-1$ and
$D'_0(w_1)=D_0(w_1)+1$, we have
$$\max\{0,(D_0-\nu_{P_0})(w_1)\}=\max\{0,(D'_0-\nu_{P'_0})(w_1)\}\;\mbox{.}$$
Therefore, in this case we also obtain that
\begin{align*}
\max&\{0,(D_0-\nu_{P_0})(w_1)\}-\max\{0,(D'_0-\nu_{P'_0})(w_1)\}\\
 &+\max\{0,(D_0-\nu_{P_0})(v_1)\}-\max\{0,(D'_0-\nu_{P'_0})(v_1)\}=
0\;\mbox{.}
\end{align*}
A symmetric argument yields that
\begin{align*}
\max&\{0,(D_0-\nu_{P_0})(w'_2)\}-\max\{0,(D'_0-\nu_{P'_0})(w'_2)\}\\
 &+\max\{0,(D_0-\nu_{P_0})(v_k)\}-\max\{0,(D'_0-\nu_{P'_0})(v_k)\}=
0\;\mbox{.}
\end{align*}
Hence,
$$\deg^+ (D_0-\nu_{P_0})-\deg^+ (D'_0-\nu_{P'_0})=0\;\mbox{.}$$
Since $\rank(D)=\deg^+ (D_0-\nu_{P_0})-1$, we have that $\rank(
D)=\deg^+ (D'_0-\nu_{P'_0})$. Since the alignment $P'_0$ can be
chosen in such a way that $\supp P'_0=B(\Gamma)\cup \supp
D'_0$ by Lemma~\ref{lm-support} and $S'<S$, the existence of $D'_0$ and
$P'_0$ contradict the choice of $D_0$ and $P_0$.

\smallskip

Last, we prove the ``furthermore'' part of the statement. That is, we show that the alignment $P$ can be
assumed to be such that all the non-branching points of $\supp P$
follow the branching points in the order determined by $P$.

Consider the alignment $P'$ obtained from $P$ by moving a point
$v\in\supp D'\setminus B(\Gamma)$ to the end of the alignment. We
claim that $\rank(D)=\deg^+(D'-\nu_{P'})-1$.

By Corollary~\ref{cor-rank}, it suffices to show that
$\deg^+(D'-\nu_{P'})\le\deg^+(D'-\nu_P)$. Let $w_1$ and $w_2$ be the
end points of the edge containing $v$. We consider in detail the
case when $v<_P w_1$ and $v<_P w_2$; the other cases are analogous.
As $D'(v)=1$ and $\nu_P(v)=-1$, it holds that $D'(v)-\nu_P(v)= 2$
and $D'(v)-\nu_{P'}(v)=0$. Similarly, $\nu_{P'}(w_i)=\nu_P(w_i)-1$,
and thus $D'(w_i)-\nu_{P'}(w_i)=D'(w_i)-\nu_P(w_i)+1$ for $i=1,2$.
We conclude that
\begin{align*}
\deg^+&(D'-\nu_P)-\deg^+(D'-\nu_{P'})=\\
=&\max\{0,D'(w_1)-\nu_P(w_1)\}-\max\{0,D'(w_1)-\nu_{P'}(w_1)\}\\
&+\max\{0,D'(w_2)-\nu_P(w_2)\}-\max\{0,D'(w_2)-\nu_{P'}(w_2)\}\\
&+\max\{0,D'(v)-\nu_P(v)\}-\max\{0,D'(v)-\nu_{P'}(v)\}\\
\ge& (D'(w_1)-\nu_P(w_1))-(D'(w_1)-\nu_{P'}(w_1))\\
&+(D'(w_2)-\nu_P(w_2))-(D'(w_2)-\nu_{P'}(w_2))+2 \ge 0\;\mbox{.}
\end{align*}
The claim now follows. Hence, we can assume without loss of
generality that all the points of $\supp D'\setminus B(\Gamma)$
follow  the points of $B(\Gamma)$ in the order determined by $P$,
i.e., $\nu_P(v)=1$ for $v\in\supp D'\setminus B(\Gamma)$.
\end{proof}

Lemma~\ref{lm-single} allows us to prove Theorem~\ref{thm:rankdeterming}. We are
grateful to an anonymous referee for suggesting the proof.
\begin{proof}[Proof
of Theorem~\ref{thm:rankdeterming}] Consider an arbitrary
divisor $D$ on a loopless metric graph $\Gamma$. By Lemma~\ref{lm-single} there exists
a rank-pair $(D',P)$ for $D$ such that $P$ is an alignment of $\supp(D')\cup B(\Gamma)$. Every edge of $\Gamma$ contains at most one inner point $v$
with non-zero value $D'(v)$, and if such a point $v$ exists, then $D'(v)=1$.
Furthermore, the points of $B(\Gamma)$
precede the other points of $\supp(D')$.

Let $E$ be the non-negative part of $D'-\nu_P$. Note that
$\supp(E)\subseteq B(\Gamma)$. We have, $D'-E\le \nu_P$. Applying the rank on
this inequality and using the fact that $\nu_P$ is non-special, we get $\rank_\Gamma(D'-E)\le \rank_\Gamma(\nu_P)=-1$. By the
choice of $D'$ and $P$, we have $\rank_\Gamma(D)=\deg(E)-1$. Consequently,
$|D'-E|=\emptyset$, that is, $E$ is as
in~\eqref{eq:defRR}, showing that $\rank_{B(\Gamma)}(D')\le
\deg(E)-1$. Consequently,
$$\rank_{\Gamma}(D)\le \rank_{B(\Gamma)}(D)= \rank_{B(\Gamma)}(D')\le
\deg(E)-1=\deg^+(D'-\nu_P)-1=\rank_\Gamma(D)\;.$$
It follows that $\rank_{\Gamma}(D)= \rank_{B(\Gamma)}(D)$, proving the theorem.  
\end{proof}

With Theorem~\ref{thm:rankdeterming} at hand, we can now give a short proof of
Theorem~\ref{thm-graph}.
\begin{proof}[Proof of Theorem~\ref{thm-graph}]
%
The following claim establishes one of the two desired inequalities between the
ranks.
\begin{claim1}\label{cl:half}
For an arbitrary divisor $F\in\Div(G)$ we have $\rank_G(F)\ge
\rank_{\Gamma}(F)$.
\end{claim1}
\begin{proof}[Proof of Claim~\ref{cl:half}]
If $F'$ is a divisor equivalent to $F$ on the graph $G$ then $F'$ is equivalent
to $F$ also on the metric graph $\Gamma$. Thus the range of minimization
in~\eqref{eq:smMin} is a superset of that in~\eqref{eq:Vecer}, and the claim
follows.
\end{proof}

Consider an arbitrary divisor $D$ on $G$. 
By Theorem~\ref{thm:rankdeterming} the set $B(\Gamma)=V(G)$ is rank-determining,
i.e., we have $\rank_\Gamma(D)=\rank_{V(G)}(D)$, where $\rank_{V(G)}(D)$ refers to $V(G)$-restricted rank on $\Gamma$. To prove the theorem, it thus
suffices to show that we range over the same sets of divisors $E$ in the
formulea~\eqref{eq:defR} (for the rank on $G$) and~\eqref{eq:defRR} (for
the $V(G)$-restricted rank on $\Gamma$). In other words, it needs to be shown
that for a divisor $E\in \Div(V(G))$ we have $|D-E|_G=\emptyset$ if and only if
$|D-E|_{\Gamma}=\emptyset$. Clearly, if $|D-E|_{\Gamma}=\emptyset$ then
$|D-E|_G=\emptyset$. The other direction follows from Claim~\ref{cl:half} which
tells us that $\rank_G(D-E)=-1$ implies $\rank_\Gamma(D-E)=-1$. 
\end{proof}

As a corollary of Theorem~\ref{thm-graph} we can prove that the rank
of a divisor on a graph is preserved under subdivision. We say that
a bijection $\varphi$ between the points of a metric graph $\Gamma$
and the points of a metric graph $\Gamma'$ is a {\em homothety} if
there exists a real number $\alpha>0$ such that $\dist_{\Gamma}(v,w)=\alpha
\cdot \dist_{\Gamma'}(\varphi(v),\varphi(w))$ for every two points
$v$ and $w$ of $\Gamma$. Note that composition of a rational
function with a homothety is a rational function, and thus a
homothety preserves the rank of divisors.

\begin{corollary}
\label{cor-graph}
Let $D$ be a divisor on a loopless graph $G$ and
let $G^k$ be the graph obtained from $G$
by subdividing each edge of $G$ exactly $k$ times.
The ranks of $D$ in $G$ and in $G^k$ are the same.
\end{corollary}

\begin{proof}
Let $\Gamma$ be the metric graph corresponding to $G$. Observe that
there exists a homothety from $\Gamma$ to the metric graph $\Gamma'$
corresponding to $G^k$. Since the rank of $D$ in $G$ is equal to the
rank of $D$ in $\Gamma$ by Theorem~\ref{thm-graph} and the rank of
$D$ in $G^k$ is equal to the rank of $D$ in $\Gamma'$ by the same
theorem, the ranks of $D$ in $G$ and in $G^k$ are the same.
\end{proof}

\medskip

Finally, we show that, in addition to the conditions given in
Lemma~\ref{lm-single}, the divisor $D'$ can be assumed to be zero
inside edges of a spanning tree of $\Gamma$. This will be used as a main
auxiliary result in Section~\ref{thm-alg} to give an algorithm for computing the
rank of divisors on metric graphs.
\begin{lemma}
\label{lm-tree} Let $D$ be a divisor on a loopless metric graph $\Gamma$. There exists a divisor $D'$, a spanning tree $T$
of $\;\Gamma$, and an alignment $P  \in \PP(\Gamma)$ such that 
$(D',P)$ is a rank-pair for $D$, $D'$ is zero in the interior of every edge of
$T$, and every edge $e\not\in T$ contains at most one interior point $v$ where $D'(v) \ne 0$, and, if such a point $v$
exists, then $D'(v)=1$.
\end{lemma}

\begin{proof}
Let $D'$ be a divisor equivalent to $D$, and let $P$ be an
alignment of the points of $B(\Gamma)\cup\supp D'$ as in
Lemma~\ref{lm-single}. 

Let us now color the edges of $\Gamma$ with red and blue, so that
the red edges contain in their interior a point $v$ in $\Gamma$ with
$D'(v)=1$ and the blue edges do not. Let $V_1,\ldots,V_k$ be
the components of $\Gamma$ formed by blue edges. Choose among all
divisors $D'$ equivalent to $D$, and alignments $P$, satisfying
the conditions of Lemma~\ref{lm-single}, the divisor $D'$ such that
the number $k$ of the components $V_1,\ldots,V_k$ is the smallest
possible. If $k=1$, there exists a spanning tree of $\Gamma$ formed
by the blue edges, and there is nothing to prove.

Assume now that $k\ge 2$ for the divisor $D'$ which minimizes $k$.
Recall that $v<_P v'$ for every $v\in B(\Gamma)$ and $v'\in\supp
D'\setminus B(\Gamma)$. Let us call the red edges
connecting points of $V_1$ to points of
$B(\Gamma) \setminus V_1$ \emph{orange} edges. We can assume that the points of $V_1 \cap
B(\Gamma)$ follow all the other points of $B(\Gamma)$ in the order
determined by $P$, as every orange edge contains a point in $\supp
D' \setminus B(\Gamma)$. For an orange edge $e$ incident with a
branching point $v_1$ of $V_1$, let $d(e)$ be the distance between
$v_1$ and the point of $\supp D'$ in the interior of $e$. Let $d_0$
be the minimum $d(e)$ taken over all orange edges $e$.

Consider the following rational function $f$:
$f(v)=0$ for points $v$ on edges between two branching points of $V_1$,
$$f(v)=\min\{d_0,\max\{0,\dist_e(v_1,v)+d_0-d(e)\}\}$$ for points $v$
on any orange edge $e$ incident with any branching point $v_1$ of $V_1$,
and $f(v)=d_0$ for the remaining points of $\Gamma$. Set
$D''=D'+\divi{f}$. Clearly, $D''$ is a divisor equivalent to $D$ that is
non-zero on at most one point in the interior of every edge of
$\Gamma$ and is equal to one at such a point. Moreover, since the
blue edges remain blue and the orange edges $e$ with $d(e)=d_0$
become blue, the number of components formed by blue edges in $D''$
is smaller than this number in $D'$.

We now find an alignment $P'$ of the points of $B(\Gamma)\cup\supp
 D''$ such that $\deg^+(D''-\nu_{P'})=\deg^+ (D'-\nu_P)$. The
existence of such an alignment $P'$ would contradict our choice of
$D'$. The alignment $P'$ is defined as follows: The branching
points of $\Gamma$ are ordered as in $P$, and they precede all the
points of $\supp D''\setminus B(\Gamma)$. The points of $\supp
 D''\setminus B(\Gamma)$ are ordered arbitrarily. If $v$ is a point
of $B(\Gamma)\cup\supp D''$ that is not contained inside an orange
edge, and that is not a branching point of $V_1$, then
$D''(v)=D'(v)$ and $\nu_{P'}(v)=\nu_{P}(v)$. Hence, such points do
not affect $\deg^+(D''-\nu_{P'})$.

We now consider a branching point $v$ of $V_1$. Let $\ell$ be the
number of edges incident with $v$ that are orange with respect to
$D'$ and blue with respect to $D''$. Clearly, $\ell$ is the number
of orange edges $e$ incident with $v$ such that $d(e)=d_0$. By the
choice of $f$, $D''(v)=D'(v)+\ell$. In addition, since the other
branching points, incident with such edges, precede $v$ in the order
determined by $P'$, $\nu_{P'}(v)=\nu_P(v)+\ell$. Hence,
$(D''-\nu_{P'})(v)=(D'-\nu_P)(v)$.

It remains to consider internal points of orange edges. Let $v$ be
such a point. If $v$ is contained in the interior of an orange edge
$e$ with respect to $D'$, then $D'(v)=1$ and $\nu_P(v)=1$, i.e.,
$(D'-\nu_P)(v)=0$. If $v$ is contained in the interior of an orange
edge $e$ with respect to $D''$, then $D''(v)=1$ and $\nu_{P'}(v)=1$,
i.e., $(D''-\nu_{P'})(v)=0$. We conclude that such points do not
affect $\deg^+(D''-\nu_{P'})$ at all. Consequently,
$\deg^+(D''-\nu_{P'})=\deg^+ (D'-\nu_P)$, as desired.
\end{proof}

\section{An algorithm for computing the rank}\label{sec-algorithm}

We now present the main algorithmic result of this paper. We
describe an algorithm which, given a metric graph $\Gamma$ and a
divisor $D$ on it, computes its rank. It is not a priori clear that such an algorithm has to exist\footnote{Indeed, let us recall as a negative
example in a similar setting that there exists no universal algorithm for solving Diophantine equations, a result due to Matiyasevich~\cite{bib-matiyasevich}.}. If the lengths of all the edges of $D$ and all the distances of non-zero
values of $D$ to the branching points are rational, then the problem
is solvable on a Turing machine. However, this need not be the case in general.
As the input can contain irrational numbers, we
assume real arithmetic operations with infinite precision to be
allowed in our computational model. The bound on
the running time of our algorithm can easily be read from its
construction; it is a simple function depending on the number of
edges, number of vertices of $\Gamma$, the ratio between the longest
and the shortest edge in $\Gamma$, and the values of $D$. The
running time is not more than exponential in any of these
parameters.

There are several papers dealing with algorithmic aspects of tropical geometry,
as~\cite{bib-bogart+,bib-jensen+,bib-theobald} for a sample. Many of these
papers rely on  machinery of commutative algebra, while our algorithm utilizes
combinatorial properties of divisors on tropical curves which were developed in
previous parts of the paper.
\begin{theorem}
\label{thm-alg}
There exists an algorithm that for a divisor $D$ on a metric graph $\Gamma$
computes the rank of $D$.
\end{theorem}
As a tool for  proving Theorem~\ref{thm-alg} we shall need the
following auxiliary result of  Gathmann and
Kerber~\cite[Lemma~1.8]{bib-gathmann+}.
\begin{lemma}\label{lem:GatAuxi}
Let a metric graph $\Gamma$ and an integer $p$ be given. Then there exists a
computable integer $U$ such that any rational function $f$ on $\Gamma$  with
$\deg^+(\divi{f})\le p$ has slope at most $U$ at every point.
\end{lemma}
\begin{remark}\label{rem:boundU}
It follows from the proof in~\cite{bib-gathmann+} that $U=(\Delta+p)^m$ in
Lemma~\ref{lem:GatAuxi} is a sufficient bound; here
$\Delta$ and $m$ are the maximum degree and the number of edges of $\Gamma$,
respectively.
\end{remark} 

We are now ready to prove Theorem~\ref{thm-alg}.
\begin{proof}[Proof of Theorem~\ref{thm-alg}]
To prove the theorem, it is enough to show that there
are only finitely many divisors equivalent to a given divisor $D$
that satisfy the conditions in the statement of Lemma~\ref{lm-tree}.

We write $\ell_e$ for the length of an edge $e$. Without loss of
generality, we can assume that $\Gamma$ is loopless, and that $\supp D\subseteq B(\Gamma)$
(introduce new branching points incident with only two edges if
needed). We can also assume that the length of each edge of $\Gamma$
is at least one. Let $n$ be the number of branching points of
$\Gamma$, $m$ the number of edges of $\Gamma$, and $M=\max_{v \:\in\:
\supp D}|D(v)|$.
We assume that $n\ge 2$ (and thus $m\ge 1$) since
otherwise $\Gamma$ is formed by a single point $w_0$ and
$\rank(D)=\max\{D(w_0),-1\}$. Similarly, we can also assume that
$M\ge 1$ since otherwise $D$ is equal to zero at all points and thus
$\rank( D)=0$. 
Finally, we let $U$ be given by Lemma~\ref{lem:GatAuxi} for
$$p=2n^2M+3mn\ge\frac12(3n^2M+3mn+m-n+1)\;.$$

We first describe the algorithm and then verify its correctness.
Fix an arbitrary vertex $w\in B(\Gamma)$. The algorithm ranges
through all spanning trees $T$ of $\Gamma$ (here, $T$ is the set of
edges of the tree, i.e., $|T|=n-1$) and all functions
$F:T\to\{-U,-U+1,\ldots,U-1,U\}$.

The algorithm then constructs all rational functions $f$ on $\Gamma$
such that for every branching point $v\in B(\Gamma)$ we have
$$f(v)=\sum_{i=1}^k F(e_i)\ell_{e_i} \;\mbox{,}$$
where $e_1, e_2, \ldots, e_k$ are the edges of $T$ on the path from
$w$ to $v$, $f$ is linear on every edge of $T$, and $\ord{v}{f}\not=0$
for at most one point $v$ on every edge not in $T$ (and $\ord{v}{f}=1$
for such a point $v$ if it exists).

Let us observe that there is only one  rational function
$f$ satisfying the above constraints. Indeed, the function $f$ is
uniquely defined on edges of $T$ as it should be linear on such
edges. Consider now an edge $e$ between branching points $v_1$ and
$v_2$ that is not contained in $T$. By symmetry, we can assume that
$f(v_1)\le f(v_2)$. Then $e$ contains a
point $v$ either with $\ord{v}{f}=1$ or $v=v_2$ such that  $f$ is linear on
$e$ everywhere except for $v$. 
The average slope of
$f$ from $v_1$ to $v_2$ along $e$ is $(f(v_1)-f(v_2))/\ell_e$. Therefore, we must have a slope of
$\lfloor(f(v_1)-f(v_2))/\ell_e\rfloor$ from $v_1$ to $v$ and a slope of
$\lfloor(f(v_1)-f(v_2))/\ell_e\rfloor+1$ from $v$ to $v_2$. This determines the
position of $v$ on $e$ as well.
We conclude that there are only
finitely many rational functions $f$ that satisfy conditions
described in the previous paragraph.

The algorithm now computes the divisor $D'=D+\divi{f}$, and then ranges
through all alignments $P$ of the points $B(\Gamma)\cup\supp D'$.
For each such alignment, the value of $\deg^+ (D'-\nu_P)-1$ is
computed and the minimum of all such values over all the choices of
$T$, $F$ (and thus $f$) and $P$ is output as the rank of $D$. Since
the number of choices of $T$, $F$ and $P$ is finite, the algorithm
eventually finishes and outputs the rank of $D$.

We have to verify that the above algorithm is correct. By
Corollary~\ref{cor-rank}, the output value is greater than or equal
to the rank of $D$. Hence, we have to show that the algorithm at
some point of its execution considers $D' \in \Div(\Gamma)$ and $P
\in \PP(\Gamma)$ such that $\deg^+(D'-\nu_P)-1=\rank(D)$. Consider
now the divisor $D'$ and the alignment $P$ as in
Lemma~\ref{lm-tree}. Since $\supp P=B(\Gamma)\cup\supp D'$, and the
algorithm ranges through all alignments $P$ of $B(\Gamma)\cup\supp
D'$ for every constructed divisor $D'$, it is enough to show that
the algorithm constructs a rational function $f$ such that
$D'=D+\divi{f}$.

Consider the step when the algorithm ranges through $T$ as in
Lemma~\ref{lm-tree} and let $f_0$ be the rational function given by
the lemma. We can assume without loss of generality that $f_0(w)=0$.

We establish that there exists a function $F:T\to\{-U,\ldots,U\}$
such that $f_0$ can be constructed (as described above) from $F$.
The existence of such a function $F$ will yield the correctness of
the presented algorithm. In order to establish the existence of $F$,
it is enough to show that absolute value of the slope of $f_0$ is
bounded by $U$ on every edge of $T$.  Due to the relation between $U$ and $p$ it suffices to prove that 
\begin{equation}\label{eq:ToEstablish}
\deg^+(\divi{f_0})\le p\;.
\end{equation}

We devote the rest of the proof to establishing~\eqref{eq:ToEstablish}.

It can be inferred from the definition of the rank that $\rank(D)
\le\deg(D)$. Hence, $\rank(D)\le nM$. We now show that $|D'(v)|\le
2(nM+m)$ for every $v\in B(\Gamma)$. If there exists a branching
point $v_0$ with $D'(v_0)> 2(nM+m)$, then 
\begin{align*}
nM &\ge \rank(D) =
\deg^+(D'-\nu_P)-1   \\
&\ge D'(v_0)-\nu_P(v_0)-1> 2(nM+m)-m-1	 \ge 2nM\;,
\end{align*}
which is
impossible. On the other hand, if there exists a branching point
$v_0$ with $D'(v_0)\le -2(nM+m)$, then $D'(v_0)-\nu_P(v_0)\le
-2(nM+m)+1<0$ and thus $\deg^+(D'-\nu_P)=\deg^+(D'')$, where
$D''(v_0)=0$ and $D''(v)=(D'-\nu_P)(v)$ for $v\not=v_0$. Observe
that
\begin{align*}
\deg(D'')&=\deg(D'-\nu_P)-(D'(v_0)-\nu_P(v_0))\\
&\ge \deg(D')-\deg(\nu_P)+2(nM+m)-1\\
&\ge -nM - (m-n) +2(nM+m)-1 = nM+m+n-1\\
&\ge r(D)+m+1\;.
\end{align*}
We therefore have
$$\deg^+(D'-\nu_P)=\deg^+(D'')\ge\deg(D'')>nM+m\ge\rank(D)+m+1,$$
which contradicts our choice of $D'$ and $P$. We conclude that
$|D'(v)|\le 2(nM+m)$, and that 
\begin{equation}\label{eq:ordFBound}
|\ord{v}{f_0}|\le M+|D'(v)|\le 3(nM+m)\;,
\end{equation}
for every $v\in B(\Gamma)$.

We express $$\deg^+(\divi{f_0})=\frac{1}{2}\sum_{v\in\Gamma}|\ord{v}{f_0}|=\frac{1}{2}\left(\sum_{v\in B(\Gamma)}|\ord{v}{f_0}|+\sum_{v\in \Gamma\setminus B(\Gamma)}|\ord{v}{f_0}|\right)\;.$$
The first sum on the right-hand side has $n$ summands, each can be bounded using~\eqref{eq:ordFBound}. The second sum has at most $m-(n-1)$ non-zero summands, each of them equal to one. Plugging in these bounds we establish~\eqref{eq:ToEstablish}.
\end{proof}

Proposition~\ref{thm-trop-reduce} and Theorem~\ref{thm-alg} now imply
the existence of an algorithm for computing the rank of a divisor on
tropical curves.

\begin{corollary}
\label{cor-alg}
There exists an algorithm that for a divisor $D$ on a tropical curve $\Gamma$
computes the rank of $D$.
\end{corollary}

The algorithm which we presented is finite, i.e., it terminates for
every input, however, its running time is exponential (as can be seen by
plugging the bound from Remark~\ref{rem:boundU} into Theorem~\ref{thm-alg}) in the size of
the input. It seems natural to ask whether it is possible to design a polynomial-time
algorithm for computing the rank of divisors. In the case of graphs the
question was posed by Hendrik Lenstra~\cite{bib-lenstra}, and, to the best of
our knowledge, is still open. Tardos~\cite{bib-tardos} presented an algorithm
which decides whether a divisor $D$ on a graph has a non-negative rank.
His algorithm is weakly polynomial, i.e., the running time is
bounded by a polynomial in the size of the graph and $\deg^+(D)$
(note that Tardos was using a different language to state the
result). It is possible to modify his algorithm in such a way that the
running time becomes polynomial in the size of the graph and
$\log(\deg^+(D))$, i.e., to obtain a truly polynomial-time
algorithm for deciding whether a given divisor on a graph has a
non-negative rank. We omit further details.

\section*{Acknowledgements}

We would like to thank Matt Baker,
Marc Coppens, and  anonymous referees for carefully reading this manuscript and
helpful comments. In particular, the current proof of Theorem~\ref{thm-graph}
via Theorem~\ref{thm:rankdeterming} was suggested by one of the referees.

\end{document}